\documentclass[a4paper,11pt]{article}
\usepackage{amsfonts,amsmath,amsthm,amssymb,graphicx,epsfig}
\usepackage[latin1]{inputenc}
\usepackage{hyperref}
\newtheorem{prop}{Proposition}
\newtheorem{theorem}{Theorem}

\newtheorem{lem}{Lemma}
\newtheorem{defi}{Definition}

\newcounter{mycount}
\newenvironment{romlist}{\begin{list}{\rm(\roman{mycount})}%
   {\usecounter{mycount}\labelwidth=1cm\itemsep -0pt}}{\end{list}}

\title{Bipartite Stable Poisson graphs on $\mathbb{R}$ }
\author{Maria Deijfen\and Fabio Marcellus Lopes}
\date{February 2012}

\begin{document}

\maketitle

\begin{abstract}
Let red and blue points be distributed on $\mathbb{R}$ according to two independent Poisson processes $\mathcal{R}$ and $\mathcal{B}$ and let each red (blue) point independently be equipped with a random number of half-edges according to a probability distribution $\nu$ ($\mu$). We consider translation-invariant bipartite random graphs with vertex classes defined by the point sets of $\mathcal{R}$ and $\mathcal{B}$, respectively, generated by a scheme based on the Gale-Shapley stable marriage for perfectly matching the half-edges. Our main result is that, when all vertices have degree 2 almost surely, then the resulting graph does not contain an infinite component. The two-color model is hence qualitatively different from the one-color model, where Deijfen, Holroyd and Peres have given strong evidence that there is an infinite component. We also present simulation results for other degree distributions.

\noindent
\vspace{0.3cm}


\vspace{0.2cm}


\end{abstract}

\section{Introduction}

Let $\mathcal{R}$ and $\mathcal{B}$ be independent homogeneous Poisson processes on $\mathbb{R}$, with finite intensities $\lambda_{\mathcal{R}}$ and $\lambda_{\mathcal{B}}$, respectively. Furthermore, let $\nu$ and $\mu$ be probability laws on the strictly positive integers and $X$ and $Y$ random variables with laws $\nu$ and $\mu$, respectively. We shall study translation invariant bipartite simple graphs whose vertices are the points in $\mathcal{R}$ and $\mathcal{B}$ and where, conditionally on $\mathcal{R}$ and $\mathcal{B}$, the degrees of the vertices are i.i.d.\ with laws $\nu$ and $\mu$, respectively. In \cite{LL}, it is shown that such graphs exist if and only if
\begin{equation}\label{des}
\lambda_{\mathcal{R}} \mathbb{E}[X]= \lambda_{\mathcal{B}} \mathbb{E}[Y].
\end{equation}
\noindent Here we focus on the question on whether there is an infinite component or not in a particular type of such graphs referred to as stable graphs. Our main result is that, if all degrees are almost surely equal to 2, then all components are finite almost surely. This implies that the percolation properties of the model are different compared to the case with a single Poisson process, where there are strong indications that there is an infinite component when all vertices have degree 2; see \cite{Mia5}. We also present simulation results for other degree distributions.

First we formally define the objects that we will work with. The \textbf{support} (or point-set) of the process $\mathcal{R}$ is the random set $[\mathcal{R}]:= \{ x \in \mathbb{R}^d : \mathcal{R}(\{x\})>0\}$ and its points are called \textbf{red points}. Analogously, $[\mathcal{B}]$ denotes the point-set of the process $\mathcal{B}$, and its points are called \textbf{blue points}. In general, for a random point measure $\Lambda$, we write $[\Lambda]$ for its support. Let $\eta_{\mathcal{R}}$ be a random integer-valued measure on $\mathbb{R}^d$, with the same support as $\mathcal{R}$, and which, conditionally on $\mathcal{R}$, assigns i.i.d.\ values with law $\nu$ to the elements of $[\mathcal{R}]$. Similarly, for $\mathcal{B}$, let $\eta_{\mathcal{B}}$ assign i.i.d.\ values with law $\mu$ to the elements of $[\mathcal{B}]$. The pairs $(\mathcal{R},\eta_{\mathcal{R}} )$ and $(\mathcal{B} ,\eta_{\mathcal{B}} )$ are independent marked Poisson processes with positive integer-valued marks. For $x \in [\mathcal{R}]$ and $y \in [\mathcal{B}]$, we write $X_{x}$ for $\eta_{\mathcal{R}}(\{x\})$ and $Y_{y}$ for $\eta_{\mathcal{B}}(\{y\})$, and we interpret this as the number of stubs (half-edges) at the red point $x$ and the number of stubs at the blue point $y$, respectively. Sometimes we refer to the stubs as red or blue depending on the color of the point to which they are attached.

A matching scheme for a marked process $(\mathcal{R},\eta_{\mathcal{R}})$ is a point process $\mathcal{M}$ on the space of unordered pairs of points in $\mathbb{R}^d$, with the property that a.s.\ for every pair $(x,y)\in \mathcal{M}$ we have that $x,y\in \mathcal{R}$, and such that in the graph with vertex set $[\mathcal{R}]$ and edge set $[\mathcal{M}]$, each vertex $x$ has degree $X_{x}$. We refer to this as a \textbf{one-color multi-matching scheme}. A matching scheme for two marked processes $(\mathcal{R},\eta_{\mathcal{R}} )$ and $(\mathcal{B} ,\eta_{\mathcal{B}} )$ is a point process $\mathcal{M}$ on the space of unordered pairs of points in $\mathbb{R}^d$ with the property that almost surely, for every unordered pair $(x,y)\in [\mathcal{M}]$, we have $x\in [\mathcal{R}]$ and $y\in [\mathcal{B}]$, and such that in the bipartite graph with vertex set $[\mathcal{R}]\cup[\mathcal{B}]$ and edge set $[\mathcal{M}]$, each vertex $x \in [\mathcal{R}]$ and each vertex $y \in [\mathcal{B}]$ has degree $X_{x}$ and $Y_{y}$, respectively. We refer to this as a \textbf{two-color multi-matching scheme}. If a vertex $x \in [\mathcal{R}]$ (or $y \in [\mathcal{B}]$) has degree at most $X_{x}$ (or $Y_{y}$), we talk about a \textbf{two-color partial multi-matching}. We only consider matching schemes $\mathcal{M}$ that are \textbf{simple} (that is, the resulting graph has no self-loops and no multiple edges) and \textbf{translation invariant} (that is, $\mathcal{M}$ is invariant in law under the action of all translations of $\mathbb{R}$). It is shown in \cite{LL} that a two-color multi-matching scheme with these properties exist if and only if (\ref{des}) holds, including also the case when $\mathbb{E}[X]=\mathbb{E}[Y]=\infty$.

One-color multi-matchings have previously been studied in \cite{Mia,Mia2,Mia5} and generalize one-color matchings, which is the particular case in which each vertex has degree 1; see \cite{Yuval1}. Two-color matchings have been studied in \cite{Yuval1} and were generalized to multi-matchings in \cite{LL}. The focus in this paper is on a particular type of two-color multi-matching referred to as stable multi-matching, inspired by Gale-Shapley stable marriage \cite{Gale}. Here and throughout, $|.|$ denotes the Euclidean norm on $\mathbb{R}$.

\begin{defi}
A one-color (two-color) multi-matching is said to be a \textbf{stable multi-matching} if almost surely, for any two distinct points $x,y$ (of different color), either they are linked by an edge or at least one of $x$ and $y$ has no incident edges longer than $|x-y|$.
\end{defi}

For the two-color case it was proved in \cite{LL}[Proposition 1] that, if both $\nu$ and $\mu$ have finite means such that (\ref{des}) holds, then there is an a.s.\ unique stable two-color multi-matching that can be obtained from the following procedure. Given the two point configurations $[\mathcal{R}]$ and $[\mathcal{B}]$, call two points $x$ and $y$ compatible if they are of different color and say that two compatible points $x$ and $y$ are mutually closest if $x$ is the closest compatible point of $y$ and vice versa. First create an edge between each pair of mutually closest compatible points in $[\mathcal{R}]\cup[\mathcal{B}]$ and remove one stub from each of these points. In the next step, consider the set of points which still have at least one stub on them after the previous step. Call two such points compatible if they are of different color and if no edge was created between them in the previous step. Create an edge between each pair of compatible mutually closest points, and remove one stub from each of these points. Iterate the algorithm indefinitely. The procedure is a straightforward generalization of the analogue construction for the one-color model.

In \cite{Mia5} it was shown for the one-dimensional one-color model that an infinite component is unique. Our first result generalizes this to the two-color case.

\begin{prop}\label{unique}
For two Poisson processes on $\mathbb{R}$ with intensities $\lambda_{\mathcal{R}}$ and $\lambda_{\mathcal{B}}$ and any two degree distributions satisfying (\ref{des}), in the two-color stable multi-matching, there is at most one infinite component.
\end{prop}

As for percolation properties, it was proved in \cite{LL} that the two-color stable multi-matching almost surely does not generate an infinite component when the only possible degrees for both processes are 1 and 2, with a strictly positive probability of degree 1 for at least one of the processes. Furthermore, for $d\geq 2$, it was shown that there is an integer $k=k(d)$ such that, if all vertices almost surely have degree at least $k$, then there is almost surely an infinite component. These results are generalizations of analogue results for the one-color case. For the one-color case there are strong indications that the case with constant degree 2 does generate an infinite component: In \cite{Mia} it is shown that almost sure existence of an infinite component follows from the assumption that a certain finite event has large enough probability, and the assumption is strongly supported by computer simulations. Our main result is that the two-color model with constant degree 2 almost surely does not give rise to an infinite component. The model is hence qualitatively different from the one-color version in this respect.

\begin{theorem}\label{noinf}
Almost surely, there is no infinite component in the two-color stable multi-matching with $\nu(\{2\})=\mu(\{2\})=1$.
\end{theorem}

The rest of the paper is organized as follows. Proposition \ref{unique} is proved in Section 2. In Section 3, we establish a consequence of percolation, more specifically, we show that almost sure existence of an infinite component in the two-color model with constant degree 2 implies that the number $N$ of vertices that would prefer to be matched to a point at the origin (in a sense defined in Section 2) is finite almost surely. Then, in Section 4, we show that in fact $N=\infty$ almost surely for any two degree distributions in the two-color model, which proves Theorem \ref{noinf}. Finally, in Section 5, simulation results for other degree distributions are presented along with some suggestions for further work.

\section{Proof of Proposition \ref{unique}}

Proposition \ref{unique} is proved along the same lines as the analog one-color result. We say that two edges $(a,b)$ and $(c,d)$ in $[\mathcal{M}]$ \textbf{cross} each other if $a<c<b<d$. Proposition \ref{unique} follows by combining the fact that two edges that cross each other in the two-color stable multi-matching on $\mathbb{R}$ must belong to the same component (Lemma \ref{unbo}) with the fact that an infinite component must be unbounded in both directions in any translation-invariant matching scheme (Lemma \ref{cross}). The latter follows from a version of the mass transport principle, which will also be used later. 

A \textbf{mass transport} is a random measure $T$ on $(\mathbb{R}^d)^2$ that is invariant in law under translations of $\mathbb{R}^d$. For Borel sets $A,B\subset \mathbb{R}^d$, we interpret $T(A,B)$ as the amount of mass transported from $A$ to $B$. Write $Q$ for the unit cube $[0,1)^d$.

\begin{lem}[Mass Transport Principle]\label{masst} Let $T$ be a mass transport. Then
\[ \mathbb{E}T(Q,\mathbb{R}^d) = \mathbb{E}T(\mathbb{R}^d,Q). \]
\end{lem}

\noindent For a proof, see e.g.\ \cite{Mia5}[Lemma 2.2].

We say that an infinite component in a translation-invariant two-color multi-matching scheme is unbounded to the right (left) with respect to the red (blue) points if, for any $r\in\mathbb{R}^+$, it contains infinitely many red (blue) points to the right (left) of $r$ ($-r$).

\begin{lem}\label{unbo}
An infinite component in a translation-invariant two-color multi-matching scheme is almost surely unbounded in both directions with respect to both colors.
\end{lem}

\begin{proof}
Just observe that a matching scheme that with positive probability is bounded in one direction with respect to (at least) one of the colors yields a contradiction with the mass transport principle: Assume for instance that an infinite component has a rightmost red point $x$. For the mass transport where each point in the component sends out mass 1 to $x$, we then have that the unit interval receives infinite mass with positive probability, while the expected mass that is sent out is bounded by $\lambda_{\mathcal{R}}+\lambda_{\mathcal{B}}$.
\end{proof}

Next we show that crossing edges in the two-color stable multi-matching cannot belong to different components. For the one-color stable multi-matching this follows immediately from the definition of the matching, but here it requires a proof. For a point $x$ in $[\mathcal{R}]$ or $[\mathcal{B}]$, write $M_x$ for the length of the longest edge incident to $x$ and say that $x$ \textbf{desires} a point $y\in\mathbb{R}$ if $|x-y|<|x-M_x|$. By the definition of the two-color stable multi-matching, if $x\in[\mathcal{R}]$ and $y\in[\mathcal{B}]$ are such that $x$ desires $y$ and $y$ desires $x$, then $x$ and $y$ must have an edge between them.

\begin{lem}\label{cross}
For the two-color stable multi-matching with any two degree distributions, crossings edges belong to the same component.
\end{lem}

\begin{proof}
Suppose that the edges $(a,b)$ and $(c,d)$ in $[\mathcal{M}]$ cross each other. By symmetry we only need to consider two cases:
\begin{romlist}
\item Assume that $a$ and $c$ are blue, and that $b$ and $d$ are hence red. Since $b-a>b-c$ and the edge $(a,b)$ is in $[\mathcal{M}]$, we have that $b$ desires $c$. Similarly, since $d-c>b-c$ and the edge $(c,d)$ is in $[\mathcal{M}]$, we also have that $c$ desires $b$. It follows from stability of the matching that $(c,b)$ is in $[\mathcal{M}]$.
\item Assume that $a$ is blue and $c$ is red, so that hence $b$ is red and $d$ is blue. Since $b-a>c-a$ and $(a,b)$ is in $[\mathcal{M}]$, we have that $a$ desires $c$. Analogously, we obtain that $d$ desires $b$. If $b-a>d-b$, then $b$ desires $d$, and the edge $(b,d)$ is in $[\mathcal{M}]$. If $b-a<d-b$, we have that $c-a<b-a<d-b<d-c$, so $c$ desires $a$, and the edge $(a,c)$ is in $[\mathcal{M}]$.
\end{romlist}
\end{proof}

\begin{proof}[Proof of Proposition \ref{unique}]
Uniqueness of the infinite component in the two-color stable multi-matching now follows from Lemma \ref{unbo} and Lemma \ref{cross} by noting that two components that are both unbounded in both directions cannot avoid having crossing edges.
\end{proof}

\section{A consequence of percolation}

For the two-color stable multi-matching, let $N$ denote the number of blue points that desire the origin. In this section we show that, if there is almost surely an infinite component in the case when all degrees are almost surely equal to 2, then $N<\infty$ almost surely. We then show in Section 4 that in fact $N=\infty$ almost surely for any degree distributions, which gives Theorem \ref{noinf}. For the one-color model, it was shown in \cite{Mia5} that existence of an infinite component is equivalent to $N<\infty$ almost surely. That model is strongly believed to percolate, but proving that $N<\infty$ almost surely has so far not been sucessful.

\begin{prop}\label{equi}
Consider the two-color stable multi-matching with $\lambda_{\mathcal{R}}=\lambda_{\mathcal{B}}$ and $\nu(\{2\})=\mu(\{2\})=1$. If almost surely there is an infinite component, then $N<\infty$ almost surely.
\end{prop}

In order to prove Proposition \ref{equi}, we need to get a grip on the structure of the components in the two-color stable multi-matching with constant degree 2. This is more involved than in the one-color model, but it turns out that an infinite component must have a relatively simple structure; see Figure 1. The following lemmas are what we need.

\begin{lem}\label{samba}
Consider the two-color stable multi-matching with $\nu(\{2\})=\mu(\{2\})=1$. For a red point $a$ and a blue point $b$, suppose that $a<b$ and that $(a,b)\in[\mathcal{M}]$. Let $c\in (a,b)$ be a blue point such that $(a,c)\not\in[\mathcal{M}]$. If $c$ has a partner $d$ to the right of $b$, then $(b,d)\in[\mathcal{M}]$. Furthermore, the point $c$ can have at most one partner to the right of $b$, and no partners to the left of $a$.
\end{lem}

\begin{proof}
As for the first claim, note that, since $b \in (c,d)$, it follows that $d$ desires $b$, and since the edge $(a,c)\not\in[\mathcal{M}]$, we must have that $c-a>d-c$. Suppose for contradiction that $b$ does not desire $d$. Then $b-a<d-b$, which implies that $d-b<d-c<c-a<b-a$, that is, $b$ desires $d$.

As for the second claim, suppose that $c$ has two partners, say $d$ and $e$, to the right of $b$. The same argument as above shows that $b$ desires both $d$ and $e$. Since $d$ and $e$ also desire $b$, the edges $(b,d)$ and $(b,e)$ must be in $[\mathcal{M}$] and hence $b$ has degree at least 3, a contradiction.

Finally, suppose that $c$ has a partner, say $f$, to the left of $a$. Since $c \in (a,b)$ and $a \in (f,c)$, it follows that $a$ and $c$ desire each other. Hence $(a,c)\in[\mathcal{M}]$.
\end{proof}

\begin{lem}\label{mia}
Consider the two-color stable multi-matching with $\nu(\{2\})=\mu(\{2\})=1$. Suppose that $(a,b)$ and $(c,d)$ in $[\mathcal{M}]$ are in the infinite component and that they cross each other, that is, $a<c<b<d$. Then $(c,b)\in \mathcal{M}$. In particular, $c$ and $b$ must be of different colors.
\end{lem}

\begin{proof}
First note that, when $c$ and $b$ are of different colors, it is clear that $c$ desires $b$ and vice versa, implying that they must be matched in $\mathcal{M}$. We are hence done if we can show that it is impossible for $c$ and $b$ to have the same color. To this end, assume that $c$ and $b$ are, say, blue, and note that both edges $(a,c)$ and $(b,d)$ cannot be in $[\mathcal{M}]$, since then the vertices $a,b,c,d$ would form a finite component of their own. So assume that $(a,c) \notin [\mathcal{M}]$. Then, by Lemma \ref{samba}, we must have $(b,d)\in [\mathcal{M}]$. Also by Lemma \ref{samba}, the second partner of $c$ (that is, not $d$) cannot be to the right of $b$ or to the left of $a$. Hence it must be inside $(a,b)$. Write $d_1$ for this second partner of $c$. Note that $d_1$, in turn, cannot have its second partner $d_2$ outside of $(a,b)$. Indeed, $d_2$ cannot be to the left of $a$, since then, by Lemma \ref{mia}, the edge $(d_2,a)$ would be in $[\mathcal{M}]$, which would imply a finite component. Furthermore, $d_2$ cannot be to the left of $b$, since that would mean that $b$ and $d_2$ desire each other and would hence be connected, which would give $b$ degree 3. Iterating this argument leads to infinitely many points inside $(a,b)$, which is impossible. The case when $(a,c)\in[\mathcal{M}]$ and $(b,d)\notin[\mathcal{M}]$ is ruled out analogously.
\end{proof}

Call a point at $x$ a \textbf{bird} if it has one partner on each side of $x$, a \textbf{left-beak} if it has both its partners to the left of $x$ and a \textbf{right-beak} if it has both its partners to the right of $x$. Lemma \ref{mia} implies that, immediately to the right of a right-beak in the infinite component, there must be a left-beak of the opposite color. In between such pairs of consecutive right-beaks/left-beaks, there may be birds; see Figure 1. The following observations are consequences of this structure.

\begin{figure}
\begin{center}
\mbox{\epsfig{file=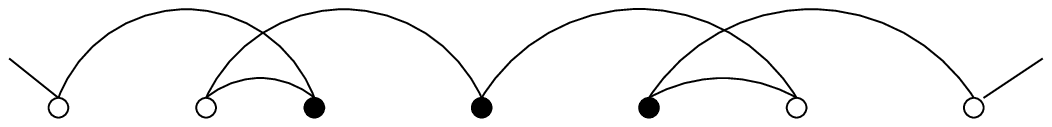}}
\end{center}
\caption{Structure of an infinite component.}
\end{figure}

\begin{lem}\label{comp}
Consider the two-color stable multi-matching with $\nu(\{2\})=\mu(\{2\})=1$ and let $\mathcal{I}$ be the set of points in the infinite component.
\begin{romlist}
\item Any point in $\mathcal{I}\cap[\mathcal{B}]$ is almost surely connected no further away from itself in a given direction than to its second nearest neighbor in $\mathcal{I}\cap[\mathcal{R}]$ in that direction.
\item Between two consecutive points in $\mathcal{I}\cap[\mathcal{B}]$, there are at most three points in $\mathcal{I}\cap[\mathcal{R}]$.
\end{romlist}
\end{lem}

\begin{proof}
As for (i), assume that a red point $a$ in the infinite component is connected to a blue point $b>a$ and that the interval $(a,b)$ contains at least two other blue points in the infinite component. First note that the path from $a$ in the direction away from $b$ cannot re-enter the interval $(a,b)$ after having left it, since that would produce a configuration that violates Lemma \ref{mia}, and analogously for the path from $b$ in the direction away from $a$. Hence each point in $\mathcal{I}\cap(a,b)$ must be hit by precisely one of these paths before the path leaves $(a,b)$ (which both paths eventually do, since there are only finitely many points in $(a,b)$). But if there are at least two blue points in the interval, at least one of the paths would generate a contradiction with Lemma \ref{mia} when leaving the interval: At most one of the blue points in $(a,b)$ can be a partner of $a$ and, immediately after hitting such a point, the path from $a$ must leave $(a,b)$ in order not to violate Lemma \ref{mia} when it does so. The rest of the blue points in $(a,b)$ must hence be on the path from $b$, and it is not hard to see that this path then yields a contradiction with Lemma \ref{mia} when it leaves $(a,b)$ (if not sooner).

As for (ii), note that, by Lemma \ref{mia}, a red left-beak in the infinite component must be followed immediately to the left by a blue right-beak and, similarly, a red right-beak must have a blue left-beak immediately to its right. In between a red left-beak and a red right-beak there can be a red bird, but no further red point without blue points in between. The maximum possible number of consecutive red points in the infinite component is hence three.
\end{proof}

\indent The next lemma is the last ingredient that we need to prove Proposition \ref{equi}. It is a generalization of \cite{Mia5}[Lemma 2.4].

\begin{lem}\label{plop}
Let $\Gamma$ and $\Lambda$ be two simple point processes on $\mathbb{R}$ with finite intensities, jointly ergodic under translations, and such that, in between two consecutive points from one of the processes, there are at most $k$ points from the other process. For $x\in [\Gamma]$, write $Z_{x}$ for the maximum of the distances from $x$ to the second nearest point of $[\Lambda]$ on the left and the second nearest point of $[\Lambda]$ on the right. The number of points $x\in [\Gamma]$ with $Z_{x}>|x|$ is finite almost surely.
\end{lem}

\begin{proof}
Consider the mass transport in which, for each $\Gamma$-point $x$, each of the intervals $(x,y)$ and $(y',x)$ between $x$ and the second nearest $\Lambda$-points $y$ and $y'$ to the right and left, respectively, sends out mass equal to $1/4k$ times its length. Note that any interval on $\mathbb{R}$ is covered by at most $4k$ such intervals -- $2k$ coming from each direction -- and hence the total mass sent out by any interval is dominated by its length. The mass is distributed uniformly over the intervals $(x-(y-x), y+(y-x))$ and $(y'-(y'-x),x+(y'-x))$, respectively. If there were infinitely many $x\in[\Gamma]$ with $Z_x>|x|$, then the unit-interval would receive infinite mass. This contradicts the mass transport principle, since the mass sent out by the unit-interval is dominated by 1.
\end{proof}

\begin{proof}[Proof of Proposition \ref{equi}]
Combining Lemma \ref{comp} and Lemma \ref{plop} -- with $\Lambda$ and $\Gamma$ in Lemma \ref{plop} corresponding to red and blue points, respectively -- yields that the number of blue points in the infinite component that desire the origin is finite almost surely. As for the finite components, it follows from Lemma \ref{cross} that edges of such components cannot cross edges of the infinite component. Hence a finite component must be contained between two consecutive points in the infinite component. Consider the longest edge in the infinite component that covers a given finite component and note that, if the finite component contains vertices that desire the origin, then at least one end-point of this covering edge must also desire the origin. Furthermore, clearly each covering edge can enclose only finitely many finite components. Hence, if with positive probability there are infinitely many finite components that contain blue vertices that desire the origin, then with positive probability there are infinitely many vertices in the infinite component that desire the origin. By symmetry, this means that there is a positive probability of having infinitely many blue points in the infinite component desiring the origin, which is a contradiction.
\end{proof}

\section{Proof of Theorem \ref{noinf}}

In this section we show that $N=\infty$ almost surely for the two-color stable multi-matching with any degree distribution(s) (satisfying (\ref{des})). Together with Proposition \ref{equi}, this proves Theorem \ref{noinf}. 

\begin{prop}\label{viola}
In the two-color stable multi-matching, for any degree distributions, we have $N=\infty$ almost surely.
\end{prop}

\begin{proof}[Proof of Theorem \ref{noinf}]
By ergodicity, the event that there exists an infinite component has probability 0 or 1. Combining Proposition \ref{equi} and Proposition \ref{viola} rules out the latter option.
\end{proof}

To prove Proposition \ref{viola}, we need a multi-matching analogue of the monotonicity property proved in \cite{Yuval1}[Lemma 17] for the two-color matching, that is, for the case with $\nu(\{1\})=\mu(\{1\})=1$. More specifically, the following lemma states that adding an extra blue point makes the multi-matching no worse for red points.

\begin{lem}[Monotonicity for the two-color stable multi-matching.]\label{beauty}
Let $U,V,\{w\}\subset \mathbb{R}^d$ be discrete disjoint sets, such that $U\cup V\cup\{w\}$ is non-equidistant and has no descending chains. For any given degree distribution(s), let $m$ be the partial two-color stable multi-matching between $U$ and $V$, and $m'$ the partial two-color stable multi-matching between $U\cup\{w\}$ and $V$. Furthermore, for $x\in U\cup V\cup\{w\}$, write $m_k(x)$ and $m'_k(x)$ for the $k$:th nearest partner of $x$ in $m$ and $m'$, respectively, with $m_k(x)=\infty$ ($m'_k(x)=\infty$) if $x$ has fewer than $k$ partners in $m$ ($m'$). Then
$$
|v-m'_{k}(v)|\leq |v-m_k(v)| \textrm{ for all $k$ and all } v\in V.
$$
\end{lem}

\begin{proof}[Proof of Lemma \ref{beauty}:]
Assume for contradiction that there is a $v\in V$ such that $|v-m'_{k}(v)|>|v-m_k(v)|$ for some $k$. This means that $v$ has at least one partner in $m$ and that at least one of these $m$-partners is not connected to $v$ in $m'$. Let $u_1$ be the closest one of the $m$-partners that $v$ is no longer connected to in $m'$. Then $v$ desires $u_1$ in $m'$, since indeed all edges of $v$ in $m'$ cannot be shorter than the distance to $u_1$. Since $v$ and $u_1$ are not connected in $m'$, it follows that $u_1$ does not desire $v$ in $m'$, that is, all edges of $u_1$ in $m'$ are shorter than the distance to $v$. Furthermore, $u_1$ has at least one partner in $m'$ that $u_1$ is not connected to in $m$ (since $u_1$ is connected to $v$ in $m$ but not in $m'$). Let $v_1$ be the closest $m'$-partner of $u_1$ that $u_1$ is not connected to in $m$. Then $|v-u_1|>|u_1-v_1|$. Since $v$ is a partner of $u_1$ in $m$, we have that $u_1$ desires $v_1$ in $m$ and, since $u_1$ and $v_1$ are not connected in $m$, it follows that $v_1$ does not desire $u_1$ in $m$. Hence all, edges of $v_1$ in $m$ are shorter than the distance to $u_1$. Furthermore, $v_1$ has at least one partner in $m$ that $v_1$ is not connected to in $m'$ (since $v_1$ is connected to $u_1$ in $m'$ but not in $m$). Let $u_2$ be the closest $m$-partner of $v_1$ that $v_1$ is not connected to in $m'$. Then $|u_1-v_1|>|v_1-u_2|$.

Iterating the above gives a sequence $v,u_1,v_1,u_2,v_2,\ldots$, such that $|v-u_1|>\ldots>|u_i-v_i|>|v_i-u_{i+1}|>\ldots$. Note that we must have $u_i\neq u_j$ for $j>i$, since $u_i$ has an edge to $v_i$ in $m'$, while all edges of $u_j$ in $m'$ are shorter than the distance to $v_{j-1}$, where $|u_j-v_{j-1}|<|u_i-v_i|$. Similarly $v_i\neq v_j$ for $j>i$. Hence $v,u_1,v_1,u_2,v_2,\ldots$ constitute a descending chain and we have arrived at a contradiction.
\end{proof}

\begin{proof}[Proof of Proposition \ref{viola}]
We have to show that the number of blue points that desire the origin is infinite almost surely. To this end, define $\widetilde{H}:=\{  x\in[\mathcal{B}]: M_{x}>|x|-1  \}$, that is, $\widetilde{H}$ is the set of blue points that desire some point in the interval $(-1,1)$, and write $\widetilde{N}$ for the cardinality of $\widetilde{H}$. The same argument as in \cite{Mia5}[Theorem 1.3(i)] applies to show that $\widetilde{N}=\infty$ a.s. implies that $N=\infty$ a.s. Hence we are done if we show that $\widetilde{N}=\infty$ a.s. The proof of this is analogue to the proof of \cite{Yuval1}[Theorem 6(i)] for two-color stable matchings:

Fix any $k<\infty$; we will prove that $\mathbb{P}(\widetilde{N}\geq k)=1$. To this end, let $(\mathcal{R}',\eta_{\mathcal{R}'})$ be obtained from $(\mathcal{R},\eta_{\mathcal{R}})$ by adding $k$ independent uniformly random red points in $(-1,1)$. By a straightforward adaptation of \cite{Yuval1}[Lemma 18(i)], and independence between $(\mathcal{R},\eta_{\mathcal{R}})$ and $(\mathcal{B},\eta_{\mathcal{B}})$, we have that the law of $(\mathcal{R}',\eta_{\mathcal{R}'},\mathcal{B},\eta_{\mathcal{B}} )$ is absolutely continuous with respect to the law of $(\mathcal{R},\eta_{\mathcal{R}},\mathcal{B},\eta_{\mathcal{B}})$. By \cite{LL}[Proposition 1], almost surely all $k$ added red points are matched in the two-color stable multi-matching $m'$ between $(\mathcal{R}',\eta_{\mathcal{R}'})$ and $(\mathcal{B},\eta_{\mathcal{B}})$. Let $m$ be the two-color stable multi-matching between $(\mathcal{R},\eta_{\mathcal{R}})$ and $(\mathcal{B},\eta_{\mathcal{B}} )$. By Lemma \ref{beauty}, the blue partners of the $k$ added red points were matched at least as far away in $m$, so these blue partners lie in $\widetilde{H}$. Thus $\widetilde{N}\geq k$, as required.
\end{proof}

\section{Simulations and further work}

\begin{table}\centering
\small{
\begin{tabular}{|c|c|c|c|c|}
\hline
\multicolumn{4}{|c|}{Number of points of each color}\\\hline
\hline
Degree &5,000 &20,000 & 50,000  \\\hline
2 & .0047 $\pm$ .0008 & .0013 $\pm$ .0017 & .0007 $\pm$ .0001  \\\hline
3 & .0716 $\pm$ .0226 & .0280 $\pm$ .0035  & .0068 $\pm$ .0066  \\\hline
4 & .5628 $\pm$ .1754 & .4124 $\pm$ .1517 & .2897 $\pm$ .1190 \\\hline
5 & .9706 $\pm$.0563 & .8920 $\pm$ .0956 & .9001 $\pm$ .1470 \\\hline
2 and 3 & .0532 $\pm$  .01809  & .0209 $\pm$ .0060 & .0109 $\pm$ .0030 \\\hline
3 and 4 & .7794 $\pm$ .1549 & .4319 $\pm$ .208& .32953 $\pm$ .1700  \\\hline
Poisson(2)+1 & .4029 $\pm$ .1505 & .2158 $\pm$ .1010 & .1668 $\pm$ .0686 \\\hline
Poisson(2)+2 & .7626 $\pm$ .2471 & .7286 $\pm$ .1949 & .7612 $\pm$ .1273  \\\hline
\end{tabular}
}
\caption{Simulation results for the two-color stable multi-matching of uniformly
random points on the cycle. The proportion of points in the largest
connected component is indicated as ``sample mean $\pm$ sample standard
deviation'' for a sample of size $10$. When there are two possible degrees, each of them has probability 0.5.}\label{sim}
\end{table}

As we have shown, there is no infinite component in the two-color stable multi-matching with constant degree 2. Is this the case also for other degree distributions? Indeed, we know that $N=\infty$ for all degree distributions. However, the method employed to show that existence of an infinite component implies that $N<\infty$ relies heavily on the fact that the structure of an infinite component can be characterized in the degree 2 case. This does not seem doable for other degree distributions, indicating that another approach is probably neccessary in the general case. The result that $N=\infty$ if all components are finite almost surely, proved for the one-color model in \cite{Mia5}, seems possible to extend to the two-color model, but does not seem to help in answering the above question.

Table 1 contains simulation results for some other degree distributions. (We have chosen to generate a fixed number of points of each color uniformly distributed on a finite cycle, which means that the points do not quite consitute two independent Poisson processes. Fixing the total number of points and then randomizing their colours would give a finite representation of our process where the points indeed constitute independent Poisson processes, but it would mean that all stubs are not matched and thus lead to difficulties in defining the largest connected component.) As for deterministic degrees, the relative size of the largest connected component seems to go to 0 for degree 3, suggesting that there is no infinite component. For constant degree 4, there is also a clearly decreasing trend, while for constant degree 5 it is difficult to draw any conclusions -- the relative size of the largest component is still large in the largest simulation, but the standard deviation is also large compared to the smaller simulations. Is it the case that there is never percolation for deterministic degrees in $d=1$? If so, does this remain true also when there is randomness in the degrees? Indeed, for a mixture of degrees 2 and 3 and of degrees 3 and 4, and for Poisson(2)+1 degrees, the relative size of the largest component seems to decrease, while for Poisson(2)+2 degrees we cannot see this in simulations of the current size.

Another question concerns the so called random direction stable multi-matching, which is obtained by assigning random directions to the stubs at the Poisson points and then do a stable matching of the stubs subject to the restriction that a given stub is to be matched in its prescribed direction. In \cite{Mia5}, it is shown that the one-color version of this model with constant degree 2 almost surely does not generate an infinite component. Is this true also in the two-color case? The two-color version of the model seems more difficult to analyze, since the component structure becomes more involved.

\end{document}